\documentclass[11pt,a4paper,oneside, english]{scrartcl}

\usepackage{amsmath}		
\usepackage{amsthm}			
\usepackage{amscd}
\usepackage{amsfonts}
\usepackage{amssymb}		
\usepackage{mathtools}		
\usepackage{rotating}
\usepackage{graphicx}
\usepackage{float}			

\usepackage[english]{babel}		
\usepackage[utf8]{inputenc}	
\usepackage{calc}
\usepackage{thumbpdf}			
\usepackage{latexsym}
\usepackage{xcolor}
\usepackage{url}				
\usepackage{cite}				

\usepackage[bookmarks,bookmarksnumbered,colorlinks]{hyperref}
\hypersetup{linkcolor = black,anchorcolor = black,citecolor =
black,filecolor = black,urlcolor = black}

\newtheoremstyle{break}
  {\topsep}{\topsep}%
  {\itshape}{}%
  {\bfseries}{}%
  {\newline}{}%
\newtheorem{thm}{Theorem}[section]

\newtheorem{lem}[thm]{Lemma}
\newtheorem{bsp}[thm]{Example}
\newtheorem{defn}[thm]{Definition}

\newtheorem{Bem}[thm]{Remark}
\newtheorem{rem}[thm]{Remark}

\makeatletter
\renewenvironment{proof}[1][\proofname]{\par
  \pushQED{\qed}%
  \normalfont\topsep6\p@\@plus6\p@\relax
  \trivlist
  \item[\hskip\labelsep
        \bfseries
    #1\@addpunct{.}]\mbox{}
}{%
  \popQED\endtrivlist\@endpefalse
}
\makeatother

\numberwithin{equation}{section}

\def\eps{\varepsilon}
\def\BB{\mathcal{B}}
\def\KK{\mathcal{K}}
\def\NN{\mathcal{N}}
\def\QQ{\mathcal{Q}}
\def\xl{x_l^\beta}
\def\ul{u_l^\beta}
\def\xb{x^\beta}

\newcommand{\R}{\mathbb{R}}
\newcommand{\N}{\mathbb{N}}		

\newcommand{\norm}[1]{\left\lVert#1\right\rVert}

\newcommand{\X}{\mathbb{X}}
\newcommand{\U}{\mathbb{U}}
\newcommand{\Y}{\mathbb{Y}}
\newcommand{\xk}{(x(k),u(k))}

\newcommand{\equb}{(x^\beta,u^\beta)}

\newcommand{\equbl}{(x^\beta_l,u^\beta_l)}
\newcommand{\tell}{\tilde{\ell}}

\begin{document}

\title{Local turnpike analysis using local dissipativity for discrete time discounted optimal control}
\author{Lars Gr\"une and Lisa Kr\"ugel}
\publishers{Chair of Applied Mathematics, Mathematical Institute\\
Universit\"at Bayreuth, Germany}
\date{\today}
\maketitle

\begin{abstract}
	Recent results in the literature have provided connections between the so-called turnpike property, near optimality of closed-loop solutions, and strict dissipativity. Motivated by applications in economics, optimal control problems with discounted stage cost are of great interest. In contrast to non-discounted optimal control problems, it is more likely that several asymptotically stable optimal equilibria coexist. Due to the discounting and transition cost from a local to the global equilibrium, it may be more favourable staying in a local equilibrium than moving to the global – cheaper – equilibrium. In the literature, strict dissipativity was shown to provide criteria for global asymptotic stability of optimal equilibria and turnpike behavior. In this paper, we propose a local notion of discounted strict dissipativity and a local turnpike property, both depending on the discount factor. Using these concepts, we investigate the local behaviour of (near-)optimal trajectories and develop conditions on the discount factor to ensure convergence to a local asymptotically stable optimal equilibrium.\\
	\textbf{Keywords: Discounted Optimal Control, Dissipativity, Turnpike}
\end{abstract}

\section{Introduction}
{\let\thefootnote\relax\footnotetext{The authors are supported by DFG Grant Gr 1569/13-2.}}
In recent years, dissipativity as introduced into systems theory by Willems \cite{Will72a,Will72b} has turned out to be a highly useful concept in order to understand the qualitative behaviour of optimally controlled systems. While related ideas were already present in early works by Willems \cite{Will71} in a linear quadratic setting, the approach has been revived and extended to fully nonlinear problems motivated by the observation of the importance of dissipativity concepts in model predictive control \cite{DiAR10,AnAR12,MuAA15,MuGA15} and for the characterization of the turnpike property \cite{GruM16}. The turnpike property expresses the fact that optimal (and possible also near-optimal) trajectories stay within a vicinity of an optimal equilibrium for most of the time. It can be seen as a way to generalize asymptotic stability properties of optimal equilibria to finite- and infinite-horizon optimal control problems.
While the references just discussed addressed non-discounted optimal control problems, the results from \cite{GGHKW18,GaGT15,GMKW20,Gruene2015} show that central results from this theory can be carried over to discounted optimal control problems and complement detectability-based approaches such as \cite{PBND17,PBND14} for analysing global asymptotic stability of equilibria of discounted optimally controlled systems.

A crucial difference between discounted and non-discounted optimal control problems is that in discounted problems it is much more likely that several asymptotically stable optimal equilibria coexist. Indeed, assuming complete controllability, in non-discounted optimal control two optimal equilibria can only coexist for arbitrary long (or infinite) horizons if they yield exactly the same optimal cost. Otherwise, for sufficiently long time it will always be beneficial to steer the system from the more expensive equilibrium to the cheaper one. In contrast to this, in discounted optimal control, due to the discounting it may not be possible to compensate for the transition cost from one equilibrium to the other with the lower cost of staying in the cheaper equilibrium. Therefore, in the discounted case locally asymptotically stable equilibria with different costs may coexist even for infinite horizon problems. In mathematical economy, where discounted optimal control problems are an important modelling tool, this is a well known fact at least since the pioneering work of Skiba \cite{Skib78} and Dechert and Nishimura \cite{DecN83}, and since then it was observed in many other papers, see, e.g., \cite{HKHF03} and the references therein.

It is the goal of this paper to show that a local version of the strict dissipativity property for discounted optimal control problems can be used for obtaining local convergence results to optimal equilibria. More precisely, we show that in the presence of local strict dissipativity and appropriate growth conditions on the optimal value functions there exist two thresholds for the discount factor $\beta\in(0,1)$, denoted by $\beta_1$ and $\beta_2$, with the following properties: Whenever $\beta\ge \beta_1$, any optimal trajectory that stays near a locally optimal equilibrium converges to this equilibrium. Whenever $\beta\le \beta_2$, any optimal trajectory that starts near this equilibrium will stay near the equilibrium. Together, this yields an interval $[\beta_1,\beta_2]$, which --- provided that $\beta_1\le \beta_2$ holds --- contains the discount factors for which convergence of optimal trajectories to the locally optimal equilibrium holds locally. 
We formalize this convergence behaviour using the formalism from turnpike theory (see, e.g., \cite{Gruene2017}), because this provides a convenient way to express these properties in a mathematically precise way also for near-optimal trajectories and to link our results to the recent literature on the relation between dissipativity and turnpike properties. We carry out our analysis in discrete time because this simplifies some of our arguments, yet we think that conceptually similar results can also be achieved for continuous time problems.

The remainder of this paper is organised as follows. In Section \ref{sec:setting} we introduce the precise problem formulation and notation. Section \ref{sec:global} summarises the known results for globally strictly dissipative discounted problems. In Section \ref{sec:local} we show how this result can be reformulated in case that only local strict dissipativity holds, provided the trajectories under consideration satisfy an invariance condition. In Section \ref{sec:stay} we then show that this invariance condition is ``automatically'' satisfied under suitable conditions. Section \ref{sec:main} then contains the main result by bringing together the two results from Sections \ref{sec:local} and \ref{sec:stay}. In Section \ref{sec:ex} we illustrate our results by several examples and the final Section \ref{sec:conclusion} provides a brief concluding discussion.

\section{Setting and preliminaries}\label{sec:setting}
\subsection{System class and notation}
We consider discrete time nonlinear systems of the form
\begin{equation}\label{eq: nsys}
x(k+1)=f(x(k),u(k)),\quad x(0)=x_0
\end{equation}
for a map $f: X\times U\to X$, where $X$ and $U$ are normed spaces. We impose the constraints  $(x,u)\in \Y\subset  X\times U$ on the state $x$ and the input $u$ and define $\X:=\{x\in X \mid \exists u\in U: (x,u)\in\Y\}$ and $\U:=\{u\in U\mid \exists x\in X: (x,u)\in \Y\}$. A control sequence $u\in \U^N$ is called admissible for $x_0\in\X$ if $\xk\in\Y$ for $k=0,\dots,N-1$ and $x(N)\in\X$. In this case, the corresponding trajectory $x(k)$ is also called admissible. The set of admissible control sequences is denoted by $\U^N(x_0)$. Likewise, we define $\U^\infty(x_0)$ as the set of all control sequences $u\in\U^\infty$ with $\xk\in\Y$ for all $k\in\N_0$. Furthermore, we assume that $\X$ is controlled invariant, i.e. that $\U^\infty(x_0)\neq \empty$ for all $x_0\in\X$. The trajectories of \eqref{eq: nsys} are denoted by $x_u(k,x_0)$ or simply by $x(k)$ if  there is no ambiguity about $x_0$ and $u$. 

We will make use of comparison-functions defined by
\begin{align*}
\mathcal{K} :=\{\alpha:\R_0^+\to\R^+_0&| \alpha \text{ is continuous and strictly increasing with }\alpha(0)=0\}\\
\mathcal{K}_\infty :=\{\alpha:\R_0^+\to\R^+_0&|\alpha\in\mathcal{K}, \alpha \text{ is unbounded}\}\\
\mathcal{L}:=\{\delta:\R_0^+\to\R^+_0&|\delta\text{ is continuous and strictly decreasing with} \lim_{t\to\infty}\delta(t)=0\}\\
\mathcal{KL}:=\{\beta:\R_0^+\times\R_0^+\to\R_0^+&|\beta \text{ is continuous, }  \beta(\cdot,t)\in\mathcal{K}, \beta(r,\cdot)\in\mathcal{L}\}.
\end{align*}
Moreover, with $\BB_\eps(x_0)$ we denote the open ball with radius $\eps>0$ around $x_0$.%

In this paper we consider infinite horizon discounted optimal control problems, i.e. problems of the type
\begin{equation}\label{dis OCP}
	\min_{u\in\U\infty(x_0)} J_\infty(x_0,u) \quad \text{with } J_\infty(x_0,u)=\sum_{k=0}^\infty \beta^k\ell(x(k,x_0),u(k)).
\end{equation}
Herein, the number $\beta \in (0,1)$ is called the discount factor. 

For such problems it was shown in \cite{GGHKW18} that if the optimal control problem is strictly dissipative at an optimal equilibrium $x^\beta$, then for sufficiently large $\beta\in(0,1)$ all optimal trajectories converge to a neighbourhood of $x^\beta$. This neighbourhood shrinks down to $x^\beta$ when $\beta\to 1$, cf.\ \cite[Theorem 4.4]{GGHKW18}. Under slightly stronger conditions on the problem data one can even show that the optimal trajectories converge to the optimal equilibrium $x^\beta$ itself and not only to a neighbourhood, cf.\ \cite[Section 6]{GGHKW18}. We will show in Theorem \ref{th: disturninf}, below, that this result can be rewritten in the language of turnpike theory, in which convergence is weakened to the property that the trajectories stay in a neighbourhood of the optimal equilibrium for a (quantifiable) amount of time, but not necessarily forever. While only the optimal trajectories satisfy convergence to the optimal equilibrium, we will show that also near-optimal trajectories satisfy the turnpike property.\footnote{We note that the turnpike property can also be defined for finite horizon optimal control problems. Still, we restrict ourselves to the infinite horizon case, since it was shown in \cite{Gruene2017} that under mild conditions on the problem data the finite horizon turnpike property holds if and only if the infinite horizon turnpike property holds.}

While this global turnpike result follows from a relatively straightforward modification of the arguments in \cite{GGHKW18}, the main question that we want to address in this paper is more difficult: assume that strict dissipativity does not hold globally but only in a neighbourhood of a locally optimal equilibrium $x_l^\beta$. Can we still expect to see a turnpike property of trajectories starting close to $x_l^\beta$?

For the derivation of our technical results, we make frequent use of the dynamic programming principle
\[V_\infty(x_0)= \inf_{u\in\U^1(x_0)}\{\ell(x,u)+\beta V_\infty(f(x_0,u))\},\]
where
\[V_\infty(x_0):=\min_{u\in\U^\infty(x_0)}J_\infty(x_0,u)\]
denotes the optimal value function of \eqref{dis OCP}.
If $u^*\in\U^\infty(x_0)$ is an optimal control sequence for an initial value $x_0\in\X$, i.e.\ if $J_\infty(x_0, u^*)=V_\infty(x_0)$ holds, then the identity 
\[V_\infty(x_0)=\ell(x_0,u^*(0))+\beta V_\infty(f(x_0,u^*(0)))\]
holds. Proofs for these statements can be found, e.g., in \cite[Section 4.2]{Gruene2017a}.
We denote optimal trajectories by $x^*(k,x_0)$ and we say that a set $\X_{inv}\subset \X$ is forward invariant for the optimally controlled system, if  for each $x_0\in\X_{inv}$ it follows that $x^*(k,x_0)\in\X_{inv}$ for all $k\geq 0$ and all optimal trajectories starting in $x_0$.

\section{The global discounted turnpike property}\label{sec:global}

In this section we first consider the optimal control problem \eqref{dis OCP} assuming {\em global} strict dissipativity. We show that under similar technical assumptions and with a similar proof technique as in \cite{GGHKW18} we can obtain a global turnpike result for near-optimal trajectories. To this end, we first introduce discounted strict dissipativity and afterwards we use it to conclude the turnpike property.

\subsection{Global discounted strict dissipativity}

We denote an equilibrium of system \eqref{eq: nsys} in the discounted case by $\equb$ since the equilibria are dependent on the discount factor $\beta\in (0,1)$.

\begin{defn}\label{def:discdiss}
Given a discount factor $\beta\in(0,1)$, we say that the system \eqref{eq: nsys} is discounted strictly dissipative at an equilibrium $\equb$ with supply rate $s:\Y\to\R$ if there exists a storage function $\lambda:\X\to\R$ bounded from below with $\lambda(x^\beta)=0$ and a class $\mathcal{K}_\infty$-function $\alpha$ such that the inequality
\begin{equation}
s(x,u)+\lambda(x)-\beta\lambda(f(x,u))\geq \alpha(\|x-\xb\|)
\label{eq:globdiss}\end{equation}
holds for all $(x,u)\in\Y$ with $f(x,u)\in\X$.
\end{defn}
The following lemma is Proposition 3.2 from \cite{GMKW20}. Since its proof is short and simple, we provide it here for convenience of the readers. It shows that we can replace the stage cost $\ell$ by a modified---usually called \emph{rotated}---stage cost $\tell$ that is positive definite without changing the optimal trajectories.

\begin{lem}\label{prop: traj}
Consider the discounted optimal control problem \eqref{dis OCP} with discount factor $\beta \in(0,1)$ and assume the system \eqref{eq: nsys} is discounted strictly dissipative at an equilibrium $\equb$ with supply rate $s(x,u)=\ell(x,u)-\ell\equb$ and bounded storage function $\lambda$.
Then the optimal trajectories of \eqref{dis OCP} coincide with those of the problem
\begin{equation}\label{mod OCP}
\min_{u\in\U^\infty(x_0)}\widetilde{J}_\infty(x_0,u) \quad\text{with } \widetilde{J}_\infty(x_0,u):=\sum_{k=0}^\infty\beta^k\tell(x(k,x_0),u(k))
\end{equation}
with rotated stage cost
\begin{equation}
\tell(x,u)=\ell(x,u)-\ell\equb+\lambda(x)-\beta\lambda(f(x,u))
\end{equation}
which is positive definite in $x^\beta$ at $\equb$, i.e.\ it satisfies the inequality $\tell(x,u) \ge \alpha(\|x-\xb\|)$ with $\alpha\in\KK_\infty$ from \eqref{eq:globdiss} for all $(x,u)\in\Y$.
\end{lem}

\begin{proof}
We rearrange
\begin{align*}
\widetilde{J}_\infty(x_0,u)&=\sum_{k=0}^\infty\beta^k\tell(x(k,x_0),u(k))\\
&=\sum_{k=0}^\infty \beta^k\left(\ell(x(k,x_0),u(k))-\ell\equb+\lambda(x(k,x_0))-\beta\lambda(x(k+1,x_0))\right)
\end{align*}
and a straightforward calculation shows that
\begin{equation}\label{eq: Jtilde}
\widetilde{J}_\infty(x_0,u)=J_\infty(x_0,u)-\dfrac{\ell\equb}{1-\beta}+\lambda(x_0)-\lim_{k\to\infty}\beta^k\lambda(x_u(k)).
\end{equation}
Since $\lambda$ is bounded and $\beta\in(0,1)$, the last limit exists and is equal to 0. Hence, the objectives differ only by expressions which are independent of $U$, from which the identity of the optimal trajectories immediately follows. The positive definiteness of $\tell$ follows from its definition, using strict dissipativity and the fact that $\lambda(\xb)=0$ implies $\tell\equb=0$. 
\end{proof}

\begin{Bem}
The requirement that $\ell\equb=0$ is the reason for imposing $\lambda(\xb)=0$ as a condition in Definition \ref{def:discdiss}. Readers familiar with dissipativity for undiscounted problems will know that in the undiscounted case $\lambda(\xb)=0$ can be assumed without loss of generality, since if $\lambda$ is a storage function then $\lambda + c$ is a storage function for all $c\in\R$. In the discounted case, this invariance with respect to addition of constants no longer holds.
\end{Bem}

\subsection{The global turnpike property}

In the non-discounted setting it is known that strict dissipativity (together with suitable regularity assumptions on the problem data) implies that optimal as well as near-optimal trajectories exhibit the turnpike property. In the discounted setting, it was observed already in \cite{Gruene2017} that for merely near-optimal trajectories the turnpike property can only be guaranteed on a finite discrete interval $\{0,\ldots,M\}$. Here $M$ depends on the deviation from optimality (denoted by $\delta$ in the following theorem) and tends to infinity as this distance tends to 0. Exactly the same happens here. As the following theorem shows, under the assumption of global discounted dissipativity we obtain precisely the turnpike property from \cite[Definition 4.2]{Gruene2017}.

\begin{thm}\label{th: disturninf}
Consider the infinite horizon optimal control problem \eqref{dis OCP} with discount factor $\beta\in(0,1)$. 
Assume that the optimal value function $\widetilde V_\beta$ of the modified problem satisfies $\widetilde V_\infty(x) \le \alpha_V(\|x-x^\beta\|)$ and
\begin{eqnarray} \widetilde V_\infty(x) \le C \inf_{u\in\U}\tilde \ell(x,u)\label{eq:Cdiss}\end{eqnarray}
for all $x\in\X$, a function $\alpha_V\in\mathcal{K}_\infty$, and a constant $C\ge 1$ satisfying 
\begin{eqnarray} C < 1/(1-\beta) \label{eq:Cbetadiss}.\end{eqnarray}%
Then the optimal control problem has the following turnpike property (cf.\ \cite[Definition 4.2]{Gruene2017}):

For each $\eps>0$ and each bounded set $\X_b\subset \X$ there exist a constant $P>0$ such that for each $M\in\N$ there is a $\delta>0$, such that for all $x_0\in\X_b$ and $u\in\U^\infty(x_0)$ with $J_\infty(x_0,u)\leq V_\infty(x_0)+\delta$, the set $\mathcal{Q}(x_0,u,\eps,M,\beta):=\{k\in\{0,\dots, M\}\mid\|x_u(k,x_0)-x^\beta\|\geq \eps\}$ has at most $P$ elements.
\end{thm}

\begin{proof}
It follows from the proof of Lemma \ref{prop: traj} that the inequality $J_\infty(x_0,u)\leq V_\infty(x_0)+\delta$ implies $\widetilde J_\infty(x_0,u)\leq \widetilde V_\infty(x_0)+\delta$. Together with the dynamic programming principle for $\widetilde V_\infty$ this yields
\begin{eqnarray*} \delta & \geq & \widetilde J_\infty(x_0,u) - \widetilde V_\infty(x_0) \\
& = & \tilde \ell(x_0,u(0)) + \beta \widetilde J_\infty(x_u(1,x_0),u(\cdot+1)) - \inf_{u\in \U}\left\{ \tilde \ell(x_0,u) + \beta \widetilde V_\infty(f(x_0,u))\right\} \\
& \geq & \tilde \ell(x_0,u(0)) + \beta \widetilde J_\infty(x_u(1,x_0),u(\cdot+1)) 
- \left( \tilde \ell(x_0,u(0)) + \beta \widetilde V_\infty(f(x_0,u(0)))\right)\\
& = & \beta(\widetilde J_\infty(x_u(1,x_0),u(\cdot+1)) - \widetilde V_\infty(f(x_0,u(0)))).
\end{eqnarray*}
This implies $\widetilde J_\infty(x_u(1,x_0),u(\cdot+1))\leq \widetilde V_\infty(x_u(1,x_0))+\delta/\beta$, and proceding inductively we obtain
\[ \widetilde J_\infty(x_u(k,x_0),u(\cdot+k))\leq \widetilde V_\infty(x_u(k,x_0))+\frac{\delta}{\beta^k} \]
for all $k\in\N$. 
This implies 
\begin{eqnarray} 
\lefteqn{\widetilde V_\infty(x_u(k+1,x_0)) - \widetilde V_\infty(x(k,x_0)) } & & \nonumber \\
& = & \frac{1}{\beta}\Big(\beta \widetilde V_\infty(x_u(k+1,x_0)) - \beta \widetilde V_\infty(x_u(k,x_0))\Big)\nonumber\\
& = & \frac{1}{\beta}\Big(\beta \widetilde V_\infty(x_u(k+1,x_0)) - \widetilde V_\infty(x_u(k,x_0)) + (1-\beta) \widetilde V_\infty(x_u(k,x_0))\Big)\nonumber\\
& \leq & \frac{1}{\beta}\Big(\beta \widetilde J_\infty(x_u(k+1,x_0),u(\cdot+k+1)) - \widetilde J_\infty(x_u(k,x_0),u(\cdot+k)) \nonumber\\
&&+ (1-\beta) \widetilde V_\infty(x_u(k,x_0))\Big) + \frac{\delta}{\beta^{k+1}}\nonumber\\
& = & \frac{1}{\beta}\Big(-\tilde\ell(x_u(k,x_0),u(k)) + (1-\beta) \widetilde V_\infty(x_u(k,x_0))\Big)+\frac{\delta}{\beta^{k+1}}\nonumber\\
& \leq & \frac{1}{\beta}\Big(-\frac{1}{C} \widetilde V_\infty(x_u(k,x_0))+ (1-\beta) \widetilde V_\infty(x_u(k,x_0))\Big) + \frac{\delta}{\beta^{k+1}}\nonumber\\
& = & \frac{\kappa}{\beta} \widetilde V_\infty(x_u(k,x_0)) +\frac{\delta}{\beta^{k+1}} \label{eq:case2}
\end{eqnarray}
where $\kappa = (1-\beta)-1/C < 0$ because of \eqref{eq:Cbetadiss}. 

This implies that for fixed $M\in\N$ and $k\in\{0,\ldots,M\}$ the function $\widetilde V_\infty$ is a practical Lyapunov function. Using \cite{Gruene2014} Theorem 2.4 restricted to $\{0,\ldots,M\}$ and the fact that $\X_b$ is bounded we can conclude that there is a sequence $\eta_k\to 0$ (depending on $\X_b$) and a function $\gamma\in\mathcal{K}_\infty$ with 
\[ \|x_u(k,x_0)-x^\beta\| \le \eta_k + \gamma(\delta/\beta^{k+1}) \leq \eta_k + \gamma(\delta/\beta^M)\]
for all $k\in\{0,\ldots,M\}$. 
This implies the desired claim by choosing $P\in\N$ (depending on $\eps$ and $\eta_k$, hence on $\X_b$) such that $\eta_k < \eps/2$ for all $k\ge P$ and $\delta>0$ (depending on $\beta$, $\eps$ and $M$) such that $\gamma(\delta/\beta^M)<\eps/2$.
\end{proof}

For an illustration of the described turnpike property we refer to Fig. \ref{figure}. We note again that in the formulation of the discounted turnpike property the level $\delta$ which measures the deviation from optimality of the trajectory $x_u(\cdot,x0)$ depends on $M$. For guaranteeing the turnpike property on $\{0,\ldots,M\}$, $\delta\to 0$ may be required if $M\to\infty$, cf.\ also Remark \ref{rem:suffcond} (iv).

\begin{figure}[htb]
\begin{center}
\includegraphics[width= 0.5\textwidth]{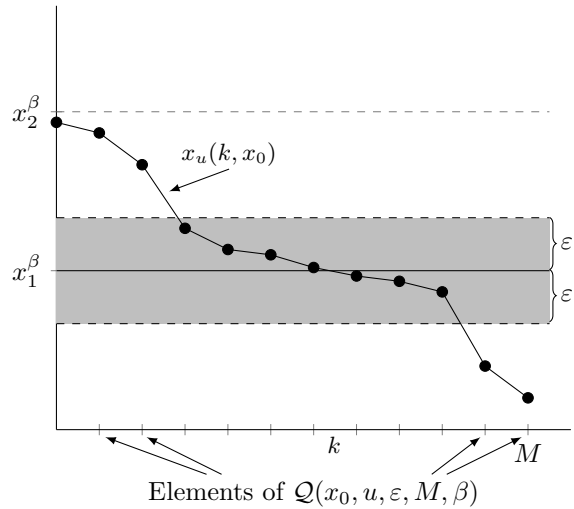}
\caption{Illustration of the set $\mathcal{Q}(x_0,u,\eps,M,\beta)$}\label{figure}
\end{center} 
\end{figure}

The following remark discusses aspects of the assumptions of Theorem \ref{th: disturninf}. For the turnpike property to hold, it is obviously necessary that the state of the system can be steered to $\xb$, at least asymptotically. This is made precise in part (i) of the remark. Part (ii) shows that if the state can be steered to $\xb$ fast enough, then a constant $C$ satisfying \eqref{eq:Cdiss} for all $\beta\in(0,1)$ holds. Finally, part (iii) of the remark discusses how inequality \eqref{eq:Cdiss} can be relaxed if such a $C$ cannot be found.

\begin{Bem} \begin{itemize}
\item[(i)] A necessary condition for the turnpike property to hold is that for each $\eps>0$, each bounded subset $\X_b\subseteq\X$ and each $x_0\in\X_b$ there exists a control sequence $u\in \U^{P+1}(x_0)$ with $x_u(k,x_0)\in \BB_\eps(\xb)$ for some $k\le P+1$, where $P$ is the constant from the turnpike property in Theorem \ref{th: disturninf}. This is immediately clear, because if such a control does not exist, then the number of points $x_u(k,x_0)\not\in\QQ(x_0,u,\eps,M,\beta)$ is larger than $P$ for all $u$.
\item[(ii)] If a constant $C$ satisfying \eqref{eq:Cdiss} for all $\beta\in(0,1)$ exists, then \eqref{eq:Cbetadiss} will hold for all sufficiently large $\beta\in(0,1)$. A sufficient condition for the existence of such a $C$ is the following exponential stabilizability assumption of the cost at the equilibrium $\equb$: there are constants $\sigma,\lambda>0$ such that for each $x_0\in\X$ there is $u\in\U^\infty(x_0)$ with
\begin{equation} \tell(x_u(k,x_0),u(k)) \le \sigma e^{-\lambda k}\inf_{\hat u \in \U} \tell(x_0,\hat u). \label{eq:expcost}\end{equation}
Then, since $\tell\ge 0$ we obtain 
\begin{align*} \widetilde V_\infty(x_0) & \le \sum_{k=0}^\infty \beta^k\tell(x_u(k,x_0),u(k)) \le \sum_{k=0}^\infty \tell(x_u(k,x_0),u(k))\\
& \le \sum_{k=0}^\infty \sigma e^{-\lambda k}\inf_{\hat u \in \U} \tell(x_0,\hat u) = \frac{\sigma}{1-e^{-\lambda}} \inf_{\hat u \in \U} \tell(x_0,\hat u), \end{align*}
implying \eqref{eq:Cdiss} with $C=\sigma/(1-e^{-\lambda})$. We note that \eqref{eq:expcost} holds in particular if the system itself is exponentially stabilizable to $\xb$ with exponentially bounded controls and $\tell$ is a polynomial\footnote{We could further relax this assumption to $\tell$ being bounded by $C_1 P$ and $C_2 P$ from below and above, respectively, for constants $C_1>C_2>0$ and a polynomial $P$.}. Exponential stabilizability of the system, in turn, follows locally around $\xb$ from stabilizability of its linearization in $\xb$. If, in addition, the necessary condition from part (i) of this remark holds, then local exponential stabilizability implies exponential stabilizability for bounded $\X$. We refer to \cite[Section 6]{GGHKW18} for a more detailed discussion on these conditions.
\item[(iii)] If a $C$ meeting \eqref{eq:Cbetadiss} and \eqref{eq:Cdiss} for all $x\in\X$ does not exist, then we may still be able to find a $C$ satisfying \eqref{eq:Cbetadiss} and \eqref{eq:Cdiss} for all $x\in\X$ with $\vartheta \le \|x-x^\beta\| \le \Theta$, for parameters $0\le\vartheta<\Theta$. In this case we can follow the reasoning in the proof of Corollary 4.3 from \cite{GGHKW18} to conclude that we still obtain a turnpike property for $\eps>\eps_0$ and $\X_b=\BB_\Delta(\xb)\cap \X$, with $\eps_0\to 0$ as $\vartheta\to 0$ and $\Delta\to\infty$ as $\Theta\to\infty$. 
\item[(iv)] Optimal trajectories, i.e., trajectories for which $J_\infty(x_0,u)=V_\infty(x_0)$ holds, satisfy the assumptions of Theorem \ref{th: disturninf} for each $\delta>0$. Hence, the assertion of the theorem holds for each $\eps>0$ and each $M\in\N$, implying that $x_u(k,x_0)$ converges to $x^\beta$ as $k\to\infty$.
\end{itemize}
\label{rem:suffcond}
\end{Bem}

\section{The local discounted turnpike property assuming invariance}\label{sec:local}
In the previous section, we have shown that an equilibrium at which the system is globally strictly dissipative has the turnpike property. Now, we consider an equilibrium denoted by $\equbl$ at which discounted strict dissipativity holds only locally, i.e., for all $x$ in a neighbourhood $\X_{\NN}$ of $\xl$, in the following sense.
\begin{defn}
Given a discount factor $\beta\in(0,1)$, we say that the system \eqref{eq: nsys} is locally discounted strictly dissipative at an equilibrium $\equbl$ with supply rate $s:\Y\to\R$ if there exists a storage function $\lambda:\X\to\R$ bounded from below with $\lambda(\xl)=0$ and a class $\mathcal{K}_\infty$-function $\alpha_\beta$ such that the inequality
\begin{equation}\label{ineq: dis.diss}
	s(x,u)+\lambda(x)-\beta\lambda(f(x,u))\geq \alpha_\beta(\|x-x_l^\beta\|)
\end{equation}
holds for all $(x,u)\in\X_\mathcal{N}\times \U$.

Further, we say that system \eqref{eq: nsys} is locally discounted strictly $(x,u)$-dissipative at the equilibrium $\equbl$ with supply rate $s:\X\times \U\to\R$ if the same holds with the inequality
\begin{equation}
	s(x,u)+\lambda(x)-\beta\lambda(f(x,u))\geq \alpha_\beta(\|(x-x_l^\beta\|+\| u-u_l^\beta)\|).
\end{equation}
\label{def:ldiss}\end{defn}
As in the global case we define the rotated stage cost by
\begin{equation}\label{eq: rotstcost b}
	\tell(x,u) := \ell(x,u)-\ell\equbl +\lambda(x)-\beta\lambda(f(x,u)).
\end{equation}%
Obviously, with this definition Lemma \ref{prop: traj} remains valid. Moreover, for $x\in\X_{\NN}$ the function $\tell$ satisfies the same properties as in the globally dissipative case. This will enable us to derive a local turnpike property, provided the neighbourhood $\X_{\NN}$ contains an invariant set $\X_{inv}\subset \X_{\NN}$ for the optimally controlled system. The following lemma gives a consequence of this assumption for the modified optimal value function, which will be important for concluding the local turnpike property.

\begin{lem}\label{lem:lbound} Consider the optimal control problem \eqref{dis OCP} with given discount factor $\beta\in (0,1)$ and assume that the system is locally strictly dissipative in $\equbl\in\X_\mathcal{N}\subset \X$. Consider a subset $\X_{inv}\subset\X_\mathcal{N}$ such that all optimal solutions $x^*(k,x_0)$ with $x_0\in\X_{inv}$ satisfy $x^*(k,x_0) \in \X_{inv}$ for all $k\ge 0$.

Then the modified optimal value function $\widetilde V_\infty$ satisfies
\begin{equation} \widetilde V_\infty(x) \ge \alpha_\beta(\|x-x_l^\beta\|) \label{eq:lbound}
\end{equation}
for all $x\in\X_{inv}$. 
\end{lem}
\begin{proof}
For all $x\in\X_\mathcal{N}$ and $u\in\U$ the modified cost satisfies $\tilde\ell(x,u) \ge \alpha_\beta(\|x-x_l^\beta\|)\ge 0$. This implies
\[ \widetilde V_\infty(x_0) = \sum_{k=0}^\infty \beta^k\tilde\ell(x(k,x_0),u(k)) \ge \sum_{k=0}^\infty \beta^k\alpha_\beta(\|x(k,x_0)-x_l^\beta\|) \ge \alpha_\beta(\|x_0-x_l^\beta\|),\]
which shows the claim.
\end{proof}

The following now gives a local version of Theorem \ref{th: disturninf}.

\begin{thm}\label{thm: localturn}
Consider the infinite horizon optimal control problem \eqref{dis OCP} with discount factor $\beta\in(0,1)$ and assume that the system is locally strictly dissipative at $\equbl\in\X_\mathcal{N}\subset \X$. Consider a subset $\X_{inv}\subset\X_\mathcal{N}$ such that all optimal solutions $x^*(k,x_0)$ with $x_0\in\X_{inv}$ satisfy $x^*(k,x_0) \in \X_{inv}$ for all $k\ge 0$ and suppose that the assumptions of Theorem \ref{th: disturninf} hold for all $x \in \X_{inv}$.

Then the optimal control problem has the following turnpike property on $\X_{inv}$:

For each $\eps>0$ and each bounded set $\X_b\subset \X_{inv}$ there exist a constant $P>0$ such that for each $M\in\N$ there is a $\delta>0$, such that for all $u\in\U^\infty(x_0)$ with $J_\infty(x_0,u)\leq V_\infty(x_0)+\delta$ and $x_u(k,x_0)\in\X_{inv}$ for all $k\in\{0,\ldots,M\}$, the set $\mathcal{Q}(x,u,\eps,M,\beta):=\{k\in\{0,\dots, M\}\mid\|x_u(k,x_0)-x^\beta\|\geq \eps\}$ has at most $P$ elements.
\end{thm}
\begin{proof}
The proof proceeds completely identical to the proof of Theorem \ref{th: disturninf}, using the fact that all inqualities used in this proof remain valid as long as the considered solutions stay in $\X_{inv}$ which is guaranteed by the assumptions. We note that Lemma \ref{lem:lbound} is needed for establishing the lower bound on $\widetilde V_\infty$ required from a practical Lyapunov function. 
\end{proof}

\begin{rem} Instead of assuming the existence of the invariant set $\X_{inv}$ we could also assume \eqref{eq:lbound} for all $x\in\X_{\mathcal{N}}$. Then by standard Lyapunov function arguments the largest sublevel set of $\widetilde V_\infty$ contained in $\X_{\mathcal{N}}$ is forward invariant for the optimal solutions and can then be used as set $\X_{inv}$. Using \eqref{eq:case2} we can even ensure that this sublevel set is also forward invariant for all solutions satisfying $J_\infty(x_0,u)\leq V_\infty(x_0)+\delta$ provided $\delta>0$ is sufficiently small. Hence, for this choice of $\X_{inv}$ the assumption that $x_u(k,x_0)\in\X_{inv}$ for all $k\in\{0,\ldots,M\}$ in Theorem \ref{thm: localturn} would be automatically satisfies if $\delta$ is not too large.
\end{rem}

\section{Optimal trajectories stay near a locally dissipative equilibrium}\label{sec:stay}

Theorem 4.3 shows that the local turnpike property holds if the optimal solutions stay in the neighbourhood of $\xl$ in which the strict dissipativity property holds. In this section we show that this condition is "automatically'' satisfied for appropriate discount factors. This will enable us to conclude a local turnpike property from local strict dissipativity. To this end, we aim to show that there exists a range of discount factors $\beta$ for which it is more favourable to stay near the locally dissipative equilibrium than to move to other parts of the state space. The first lemma we need to this end shows a property of trajectories that move out of a neighbourhood of $\xl$. In contrast to the previous result, now we need the stronger $(x,u)$-dissipativity.

\begin{lem}\label{lem: ball}
	Consider a discounted optimal control problem \eqref{dis OCP} subject to system \eqref{eq: nsys} with continuous $f$. Assume local strict $(x,u)$-dissipativity at an equilibrium $\equbl$ according to Definition \ref{def:ldiss} and let $\rho>0$ be such that $\mathcal{B}_\rho(x_l^\beta)\subset\X_{\NN}$ holds for the neighbourhood $\X_{\NN}$ from Definition \ref{def:ldiss}. Then there exists $\eta >0$ such that for each $K\in\N$ and any trajectory $x(\cdot)$ with $x_0=x(0)\in\BB_\eta(\xl)$ and $x(K)\notin \mathcal{B}_\rho(x_l^\beta)$ there is a $M\in\{0,\dots, K-1\}$ such that $x(0),\ldots,x(M)\in \mathcal{B}_\eta(\xl)$ and either 
\[
	\text{(i) } x(M)\in\mathcal{B}_\rho(x_l^\beta)\backslash\mathcal{B}_\eta(x_l^\beta)
\quad \text{or} \quad
\text{(ii) } \|u(M)-u_l^\beta\|\geq \eta\]
holds.
\end{lem}

\begin{proof}
	The continuity of $f$ implies that there exists $\eps>0$ such that $\|f(x,u)-x_l^\beta\|<\rho$ for all $(x,u)\in\Y$ with $\|x-x_l^\beta\|<\eps$ and $\|u-u_l^\beta\|<\eps$. We let $K_{\min}$ be minimal with $x(K_{\min})\notin\mathcal{B}_\rho(x_l^\beta)$, set $\eta:=\min\{\eps, \rho\}$, and claim that this implies the assertion for $M=K_{\min}-1$.
	
	We prove this claim by contradiction. To this end, we assume that for $M=K_{\min}-1$ neither assertion (i) nor assertion (ii) holds. This implies on the one hand that $\|x(M)-x_l^\beta\|<\eta$, since $x(M)\in\BB_\rho(\xl)$ by minimality of $K_{\min}$ and (i) is not fulfilled. On the other hand, it implies $\|u(M)-u_l^\beta\|<\eta$, because (ii) does not hold. Then, however, since $\eta \le \eps$, the continuity of $f$ implies
	\[\|x(K_{\min})-x_l^\beta\| =  \|f(x(M),u(M))-x_l^\beta\|< \rho.\]
	This means that $x(K_{\min})\in\BB_\rho(\xl)$, which is a contradiction to the choice of $K_{\min}$. Hence, either assertion (i) or assertion (ii) must hold for $M=K_{\min}-1$.
\end{proof}

The next lemma shows that the behavior characterized in Lemma \ref{lem: ball} induces a lower bound for the rotated discounted functional $\widetilde{J}_\infty$ from \eqref{mod OCP} along trajectories that start in a neighborhood of $\xl$ and leave this neighborhood. To this end, we note that even if merely local strict dissipativity holds, the modified stage cost $\tell$ from \eqref{eq: rotstcost b}  is well defined, since $\lambda$ is defined for all $x\in\X$. However, the inequality $\tell(x,u) \ge \alpha_\beta(\|x-\xl\|+\|u-\ul\|)$ and, more generally, positivity of $\tell$ are only guaranteed for $x\in\X_{\NN}$. 

\begin{lem}\label{lem: optimal}
Let the assumptions of Lemma \ref{lem: ball} hold. In addition, assume that $\lambda$ from Definition \ref{def:ldiss} is bounded and the stage cost $\ell$ is bounded from below. Then, there exists $\beta^\star\in(0,1)$ with the following property: for any $\beta\in(0,\beta^\star)$ and any $K\in\N$ there is $\sigma(\beta,K)>0$ such that for any trajectory $x(\cdot)$ with $x_0=x(0)\in\BB_\eta(\xl)$ and $x(P)\notin \mathcal{B}_\rho(x_l^\beta)$ for some $P\in\{1,\ldots,K\}$ the inequality
\begin{equation} \widetilde J_\infty(x_0,u) \ge \sigma(\beta,K) \label{ineq: goal}\end{equation}
holds.
\end{lem}

\begin{proof}
First observe that boundedness from below of $\ell$ and boundedness of $\lambda$ imply boundedness from below of $\tell$. Let $\tell_{\min} := \inf_{(x,u)\in\Y} \tell(x,u)$ with $\tell_{\min}<0$. Moreover, local dissipativity implies that $\tell(x,u)\ge 0$ for all $x\in\X_{\NN}$ and all $u\in\U$. 

    Since the trajectory under consideration satisfies the assumptions of Lemma \ref{lem: ball} with $K=P$, there exists $M\in\{0,\dots,P\}$ such that either assertion (i) or assertion (ii) of this lemma holds. In case (i), we obtain that 
    \[ \tell(x(M),u(M)) \ge \alpha_\beta(\|x(M)-\xl\|) \ge \alpha_\beta(\eta)  \]
    and in case (ii) we obtain 
    \[ \tell(x(M),u(M)) \ge \alpha_\beta(\|u(M)-\ul\|) \ge \alpha_\beta(\eta).  \]
    Hence, we get the same inequality in both cases and we abbreviate $\delta:=\alpha_\beta(\eta)>0$. 
    In addition, Lemma \ref{lem: ball} yields $x(0),\ldots,x(M)\in \mathcal{B}_\eta(\xl)\subset\X_{\NN}$, which implies $\tell(x(k),u(k))\ge 0$ for all $k=0,\ldots,M-1$, and the lower bound on $\tell$ implies $\tell(x(k),u(k))\ge\tell_{\min}$ for all $k\ge M+1$. Together this yields
	\begin{align*} \widetilde J_\infty(x_0,u)  = & \sum_{k=0}^\infty\beta^k\ell(x(k,x_0),u(k)\\
		=&\sum_{k=0}^{M-1}\beta^k\underbrace{\ell(x(k,x_0),u(k))}_{\ge 0}+\beta^M\underbrace{\ell(x(M,x_0),u(M))}_{\ge \delta}+\sum_{k=M+1}^\infty\beta^k\underbrace{\ell(x(k,x_0),u(k))}_{\ge \tell_{\min}}\\
		\geq& \beta^M\delta+\dfrac{\beta^{M+1}}{1-\beta}\tell_{\min} = \dfrac{\beta^M}{1-\beta}\left(\left(\tell_{\min}-\delta\right)\beta+\delta\right).
	\end{align*}
We now claim that the assertion holds for $\sigma = \frac{\beta^K\delta}{2(1-\beta)} \le \frac{\beta^M\delta}{2(1-\beta)}$. To this end, it is sufficient to show the existence of $\beta^\star$ with 
\[ \dfrac{\beta^M}{1-\beta}\left(\left(\tell_{\min}-\delta\right)\beta+\delta\right)\ge \frac{\beta^M\delta}{2(1-\beta)}\]
for all $\beta\in(0,\beta^\star)$. 
This is equivalent to 
\[ \dfrac{\beta^M}{1-\beta}\left(\left(\tell_{\min}-\delta\right)\beta+\frac{\delta}{2}\right)\ge 0\;\; \Leftrightarrow \;\; \left(\tell_{\min}-\delta\right)\beta+\frac{\delta}{2}\ge 0,\]
since $\tell_{\min}-\delta<0$. This inequality holds for all $\beta\in(0,\beta^\star)$ if $\beta^*=\delta/(2(\delta-\tell_{\min}))$.

\end{proof}

\begin{rem}
The choice of the fraction $\frac 1 2$ for $\sigma$ in the proof of Lemma \ref{lem: optimal} is arbitrary. We can also use a more general fraction $\frac{1}{k+1}$ with $k\in\N$. Then, with the same calculation as above we get that $\beta^* = \dfrac{k}{k+1}\dfrac{\delta}{\delta-\tell_{\min}}$.
\end{rem}

Based on the estimate from Lemma \ref{lem: optimal} we can now conclude that near-optimal solutions starting near $\xl$ stay in $\X_\mathcal{N}$ for a certain amount of time.

\begin{lem} Consider a discounted optimal control problem \eqref{dis OCP} subject to system \eqref{eq: nsys} with $f$ continuous and stage cost $\ell$ bounded from below. Assume local strict $(x,u)$-dissipativity at an equilibrium $\equbl$ according to Definition \ref{def:ldiss} with bounded storage function $\lambda$. Assume furthermore that there is $\gamma\in\KK_\infty$ and $\hat\beta\in(0,1]$ such that $|\widetilde V_\infty(x)| \le \gamma(\|x-\xl\|)$ for all $x\in\X_{\NN}$ and all $\beta\in(0,\hat\beta]$. Then there exists $\beta_2\in(0,1)$ with the following property: for any  $\beta\in(0,\beta_2)$ and any $K\in\N$ there exists a neighbourhood $\BB_{\eps(\beta,K)}(\xl)$ and a threshold value $\theta(\beta,K)>0$ such that all trajectories with $x_0\in \BB_{\eps(\beta,K)}(\xl)$ and $J_\infty(x_0,u) < V_\infty(x_0) + \theta(\beta,K)$ satisfy $x(k)\in \X_{\NN}$ for all $k\in\{0,\ldots,K\}$.
\label{lem:stay}\end{lem}

\begin{proof} We choose $\beta_2$ as the minimum of $\beta^\star$ from Lemma \ref{lem: optimal} and $\hat\beta$. We further use $\sigma(\beta,K)>0$ from Lemma \ref{lem: optimal} to set $\eps(\beta,K):= \gamma^{-1}(\sigma(\beta,K)/2)$ and $\theta(\beta,K):= \sigma(\beta,K)/2$. Now consider a trajectory meeting the assumptions and observe that since $J_\infty$ and $\widetilde J_\infty$ differ only by a term that is independent of $u(\cdot)$, the assumption $J_\infty(x_0,u) \le V_\infty(x_0) + \theta(\beta,K)$ together with the assumption on $x_0$ implies 
\[ \widetilde J_\infty(x_0,u) < \widetilde V_\infty(x_0) + \theta(\beta,K) < \gamma(\eps(\beta,K)) +  \theta(\beta,K).\]
The definition of $\theta$ and $\eps$ then implies
\[\widetilde J_\infty(x_0,u) < \sigma(\beta,K)/2 + \sigma(\beta,K)/2 = \sigma(\beta,K). \]
Since by Lemma \ref{lem: optimal} any trajectory leaving $\X_{\NN}$ (and thus also $\BB_\rho(\xl)$) up to time $K$ has a rotated value satisfying
\[\widetilde J_\infty(x_0,u) \ge \sigma(\beta,K), \]
the trajectory under consideration cannot leave $\X_{\NN}$ for $k\in\{0,\ldots,K\}$.
\end{proof}

\section{The local discounted turnpike property without assuming invariance}
\label{sec:main}

With the preparations from the previous sections, we are now able to formulate our main theorem on the existence of a local turnpike property.

\begin{thm}\label{th: locturnpike}
Consider a discounted optimal control problem \eqref{dis OCP} subject to system \eqref{eq: nsys} with $f$ continuous and stage cost $\ell$ bounded from below. Assume local strict $(x,u)$-dissipativity at an equilibrium $\equbl$ according to Definition \ref{def:ldiss} with bounded storage function $\lambda$ on $\X_{\NN}$. Assume furthermore that there is $\gamma\in\KK_\infty$ and $\hat\beta\in(0,1]$ such that $|\widetilde V_\infty(x)| \le \gamma(\|x-\xl\|)$ for all $x\in\X_{\NN}$ and all $\beta\in(0,\hat\beta)$, and that there is an interval $[\beta_1,\beta^*]$ of discount rates with $\beta_1<\hat\beta$, such that for each $\beta\in(\beta_1,\beta^*)$ the assumptions of Theorem \ref{th: disturninf} hold for all $x \in \X_{\NN}$.

Then there is $\beta_2\in(0,1)$ such that for all $\beta\in (\beta_1,\beta_2)$ there exists a neighbourhood $\NN$ of $\xl$ on which the system exhibits a local turnpike property in the following sense: 

For each $\eps>0$ there exist a constant $P>0$ such that for each $M\in\N$ there is a $\delta>0$, such that for all $x_0\in \NN$ and all $u\in\U^\infty(x_0)$ with $J_\infty(x_0,u)\leq V_\infty(x_0)+\delta$, the set $\mathcal{Q}(x,u,\eps,M,\beta):=\{k\in\{0,\dots, M\}\mid\|x_u(k,x_0)-x^\beta\|\geq \eps\}$ has at most $P$ elements. 

Particularly, if $J_\infty(x_0,u) = V_\infty(x_0)$, i.e., if the trajectory is optimal, then for each $\eps>0$ the set $\mathcal{Q}(x,u,\eps,\infty,\beta):=\bigcup_{M\in\N}\mathcal{Q}(x,u,\eps,M,\beta)$ has at most $P$ elements, implying the convergence $x_u(k,x_0)\to x^\beta$ as $k\to\infty$.
\end{thm}
\begin{proof} The idea of the proof is to use $\beta_2$ from Lemma \ref{lem:stay} and, for each $\beta\in(\beta_1,\beta_2)$, to construct a neighbourhood $\NN$ of $\xl$ and a $\delta>0$ such that all trajectories starting in $x_0\in\NN$ and satisfying $J_\infty(x_0,u)\leq V_\infty(x_0)+\delta$ stay in $\NN$ for all future times. Then the turnpike property follows from Theorem \ref{thm: localturn} applied with $\X_{inv}=\NN$.

To this end, we take $\beta_2$ from Lemma \ref{lem:stay}, fix $\beta\in(\beta_1,\beta_2)$, and consider the neighbourhood $\BB_{\eps(\beta,1)}(\xl)$ and the threshold value $\theta(\beta,1)$ from Lemma \ref{lem:stay} for $K=1$. We choose $\NN$ as the largest sublevel set of $\widetilde V_\infty$ that is contained in $\BB_{\eps(\beta,1)}(\xl)$ and denote the level by $\lambda>0$, i.e., $\NN=\{x\in\X_{\NN}\,|\, \widetilde V_\infty(x) < \lambda\}$. We abbreviate $\kappa = (1-\beta)-1/C$, observing that $\kappa<0$ because of because of \eqref{eq:Cdiss} (cf.\ also the proof of Theorem \ref{th: disturninf}), and set
\[ \delta:= \beta^M\min\left\{ \theta(\beta,1), -\frac{\kappa\lambda}{2\beta},\frac \lambda 2\right\}. \]
Now let $x_0$ and $u$ be as in the assertion, i.e., satisfying $J_\infty(x_0,u)\leq V_\infty(x_0)+\delta$, and denote the corresponding trajectory by $x(\cdot)$. Then, just as in the first part of the proof of Theorem \ref{th: disturninf}, we obtain the estimate
\[ \widetilde J_\infty(x(k),u(\cdot+k))\leq \widetilde V_\infty(x(k))+\frac{\delta}{\beta^k} \]
for all $k\in\N$. By definition of $\delta$ this in particular implies
\begin{equation} \widetilde J_\infty(x(k),u(\cdot+k))\leq \widetilde V_\infty(x(k))+ \theta(\beta,1) 
\label{eq:Jbound} \end{equation}
for all $k=0,\ldots,M$. 

Now we prove by induction that $x(k)\in \NN$ for all $k=0,\ldots,M$. For $k=0$ this follows from the choice of $x_0$. For $k\to k+1$, we make the induction assumption that $x(k)\in \NN$, i.e., $\widetilde V_\infty(x(k))<\lambda$. Then, because of \eqref{eq:Jbound} and $\NN\subseteq \BB_{\eps(\beta,1)}(\xl)$, Lemma \ref{lem:stay} (applied with initial value $x_0=x(k)$ and control $u(\cdot+k)$) implies that $x(k+1)\in \X_{\NN}$. Hence, all the (in)equalities leading to inequality \eqref{eq:case2} in the proof of  Theorem \ref{th: disturninf} are valid and, together with the definition of $\delta$, yield
\[ \widetilde V_\infty(x(k+1)) - \widetilde V_\infty(x(k)) \le \frac{\kappa}{\beta} \widetilde V_\infty(x(k)) + \frac{\delta}{\beta^k} \le \frac{\kappa}{\beta} \widetilde V_\infty(x(k)) + \min\left\{\theta(\beta,1),  -\frac{\kappa\lambda}{2\beta},\; \frac \lambda 2\right\}. \]

Now if $\widetilde V_\infty(x(k)) \ge \lambda/2$, then second term in the minimum defining $\delta$ implies 
\[ \widetilde V_\infty(x(k+1)) - \widetilde V_\infty(x(k)) \le  \frac{\kappa}{\beta} \frac \lambda 2 -\frac{\kappa\lambda}{2\beta}=0,\]
implying $\widetilde V_\infty(x(k+1)) \le \widetilde V_\infty(x(k)) < \lambda$ and thus $x(k+1)\in \NN$. 

If $\widetilde V_\infty(x(k)) < \lambda/2$, then the third term in the minimum defining $\delta$ implies 
\[ \widetilde V_\infty(x(k+1)) - \widetilde V_\infty(x(k)) \le  \underbrace{\frac{\kappa}{\beta} \widetilde V_\infty(x(k)}_{\le 0} + \frac \lambda 2\le \frac  \lambda 2,\]
implying $\widetilde V_\infty(x(k+1)) \le \widetilde V_\infty(x(k)) + \frac{\lambda}{2} < \lambda$, i.e., again $x(k+1)\in \NN$. This proves the induction step and hence $x(k)\in \NN$ for all $k=0,\ldots,M$. 

Now the turnpike property follows from Theorem \ref{thm: localturn} applied with $\X_{inv}=\NN$.\end{proof}

\begin{rem}
We note that the interval $(\beta_1,\beta_2)$ may be empty. This is because 
\begin{enumerate}
\item[(i)] the condition \eqref{eq:Cdiss} needed for proving the turnpike property for trajectories staying near $\xl$ may require sufficiently large $\beta$ to hold
\item[(ii)] a trajectory starting near $\xl$ will in general only stay near $\xl$ for sufficiently small $\beta$
\end{enumerate}
More precisely, the lower bound in (ii) as identified at the end of the proof of Lemma \ref{lem: optimal} depends on the cost $\tell$ outside a neighbourhood of $\xl$ and the cost to leave this neighbourhood. The upper bound in (i), in turn, depends on the cost to reach the equilibrium $\xl$ from a neighbourhood. If this cost is high and, in addition, the cost to leave the neighbourhood and the cost outside the neighbourhood are low, then the set of discount rates for which a local turnpike behaviour occurs may be empty. 
\end{rem}

\begin{rem}
The attentive reader may have noted that we apply Lemma \ref{lem:stay} with $K=1$ in this proof, rather than with $K=M$, which might appear more natural given that we want to make a statement for $\{0,\ldots,M\}$. This is because the size of the neighbourhood $\BB_{\eps(\beta,K)}(\xl)$ delivered by Lemma \ref{lem:stay} depends on $K$. Hence, if we applied Lemma \ref{lem:stay} with $K=M$ in order to construct the neighbourhood $\NN$, this neighbourhood may shrink down to $\{\xl\}$ as $M$ increases. In contrast to this, the fact that $\widetilde V_\infty$ is a (practical) Lyapunov function allows us to construct a neighbourhood $\NN$ that does not depend on $M$. 
\end{rem}

\section{Examples}\label{sec:ex}

We end our paper with a couple of examples illustrating our theoretical results.  All numerical solutions were obtained using a dynamic programming algorithm as described in \cite{GruS04}. We start with two examples exhibiting a locally and a globally optimal equilibrium.

\begin{bsp}\label{ex:1}
Consider the dynamics $f(x,u) = x+ u$ and the stage cost $\ell(x,u)=x^4-\frac 1 4 x^3 - \frac 7 4 x^2$. 
	\begin{figure}[htb]
	\begin{center}
		\includegraphics[width= 0.3\textwidth]{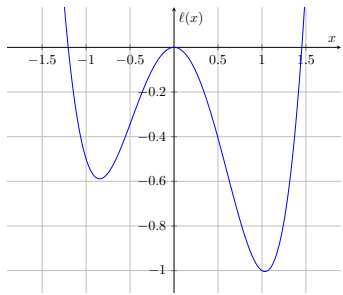}
		\caption{Stage cost $\ell(x)$}\label{fig: ex sc}
	\end{center} 
\end{figure}
As visualized in Figure \ref{fig: ex sc}, the stage cost $\ell$ has a local minimum in $x =  \frac{3-\sqrt{905}}{32}$, a maximum in $x=0$ and a global minimum in $x = \frac{3+\sqrt{905}}{32}$. Following \cite[Section 4]{GMKW20} we can calculate the storage function $\lambda$ by using the optimality conditions for optimal equilibria. We remark that the procedure for computing global storage functions described in this reference also works for the local dissipativity in case of local convexity which is given in this example, cf.\ also the discussion after Example \ref{ex:2}, below. Thus, by a straightforward calculation, we get the local equilibrium $\equbl =  (\frac{3-\sqrt{905}}{32},0)$ and the storage function $\lambda\equiv 0$. Inserting this, we get the rotated stage cost $\tell(x,u) = x^4-\frac 1 4 x^3 - \frac 7 4 x^2 - \ell(\xl, 0)$ and local discounted strict $(x,u)$-dissipativity of the system $f(x,u) = x+ u$ at $\xl$ for any $\beta\in(0,1)$. Thus, the assumptions of Lemma \ref{lem: ball} and Lemma \ref{lem: optimal} are fulfilled. Hence, following the proof of Lemma \ref{lem: optimal} we can estimate $\beta_2\approx 0.67$ with $\delta \approx 1$ and $\tell_{\min}\approx -0.42$. Further, since $\norm{\tell(x,u)}$ is bounded for $x$ in a neighbourhood $\BB_\eps(x_0)$, $\eps>0$, Theorem \ref{th: locturnpike} can be applied. For illustrating the theoretical results, we set $\U=[-0.75, 0.75]$.

\begin{figure}[htb]
	\begin{center}
	\begin{minipage}[t]{0.47\textwidth}
		\includegraphics[width= \textwidth]{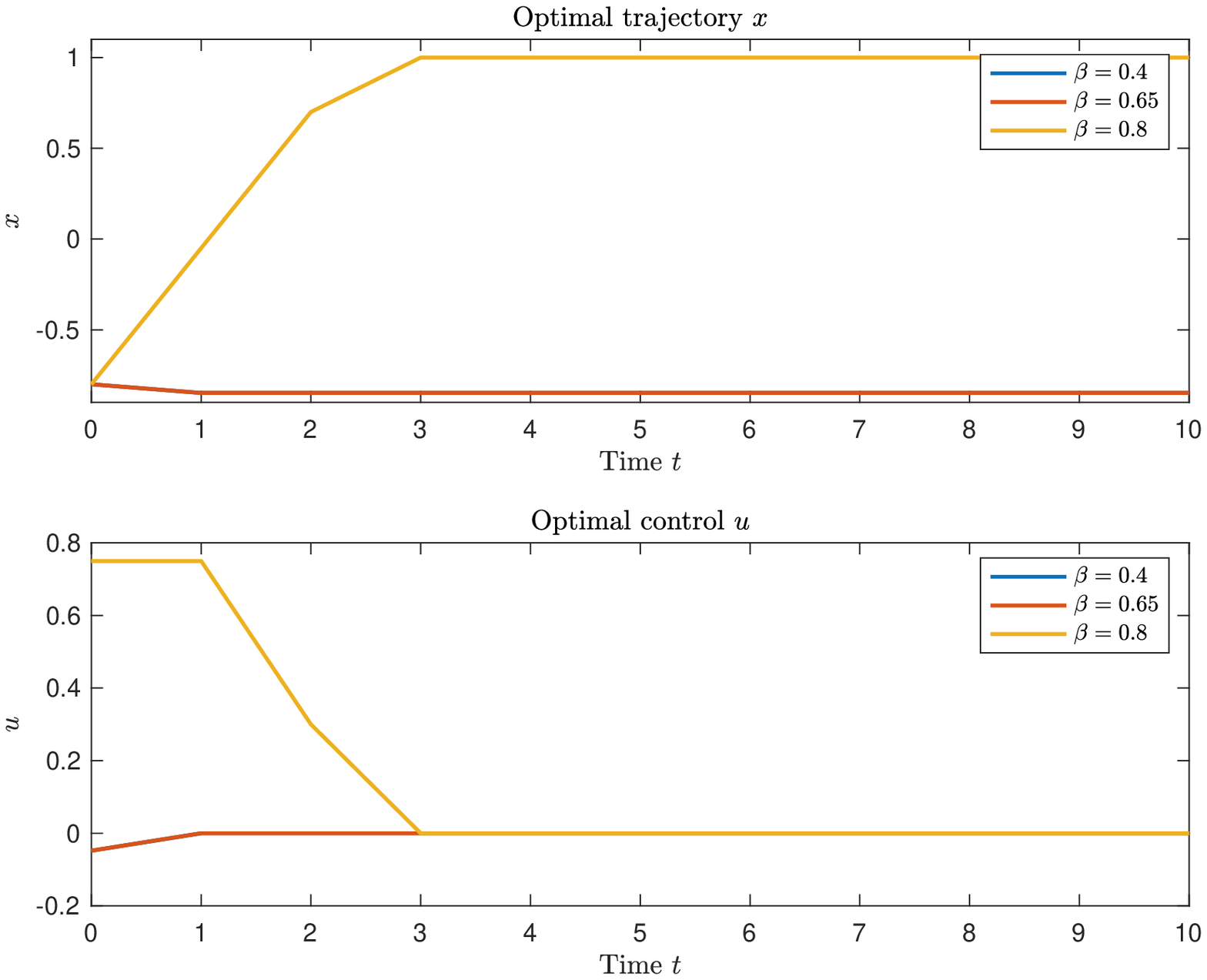}
	\end{minipage}
	\begin{minipage}[t]{0.47\textwidth}
	\includegraphics[width = \textwidth]{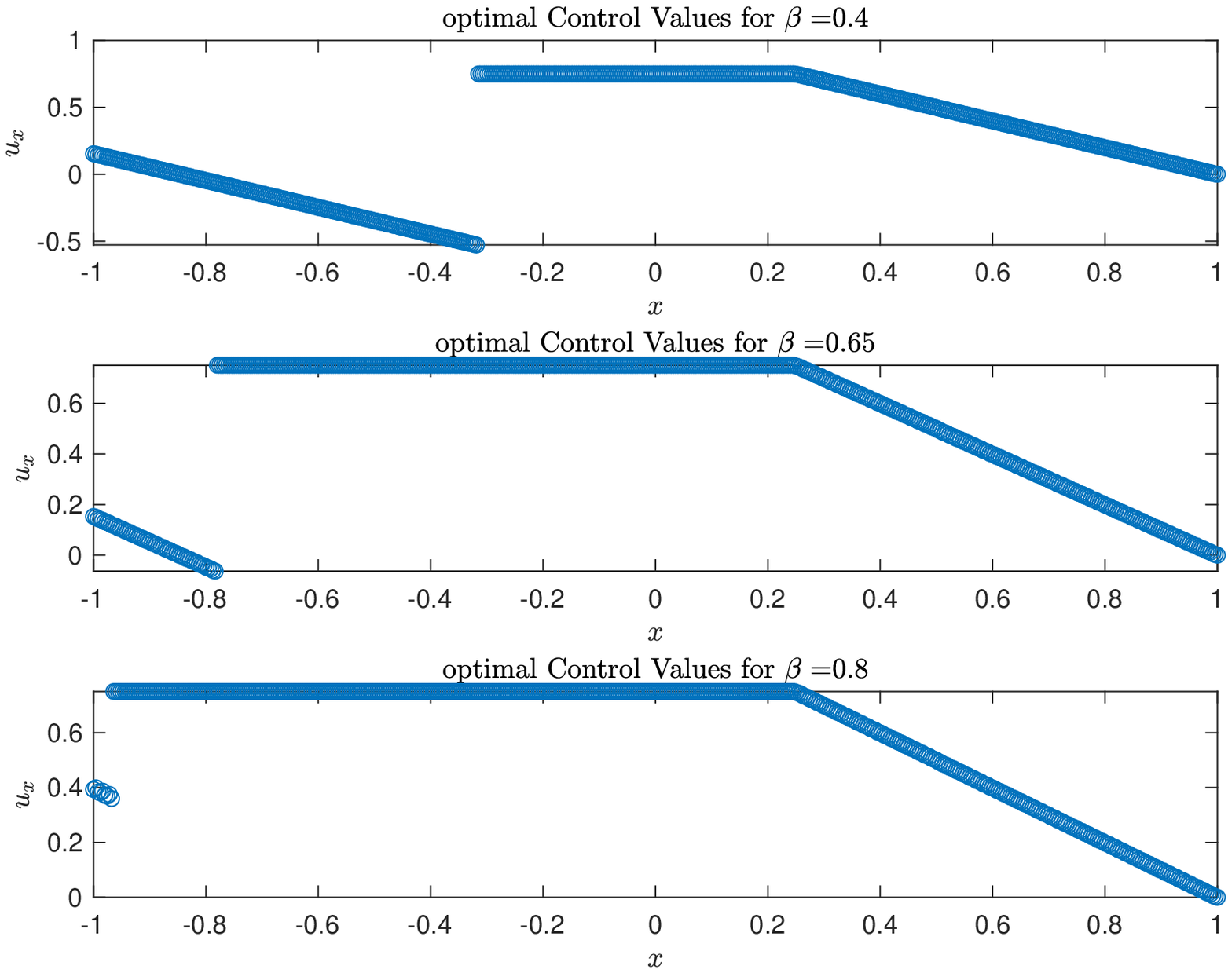}
\end{minipage}
		\caption{Example \ref{ex:1} with $x_0 = -0.8$}\label{fig: ex1 beta}
	\end{center} 
\end{figure}

On the left hand side of Figure \ref{fig: ex1 beta} we show the behaviour of the trajectory $x$ and the control $u$ for different discount factors $\beta$. On the right hand side, we can observe the optimal feedback control values $u_x$ and therefore the domain of attraction of the equilibria dependent on $\beta$. After a maximum of three time instants, the trajectory reaches the global equilibrium for $\beta$ large enough. In contrast, for $\beta\leq 0.67$ we can observe that it is more favourable to stay in a neighbourhood of the local equilibrium $\xl$. We remark that is sufficient to depict $\beta = 0.8$ as a representative for all $\beta \in (0.67, 1)$ since the behaviour of the trajectory, the control and the stage cost does not change significantly.

\begin{figure}[htb]
	 \begin{center}
	\begin{minipage}[t]{0.47\textwidth}
		\includegraphics[width=\textwidth]{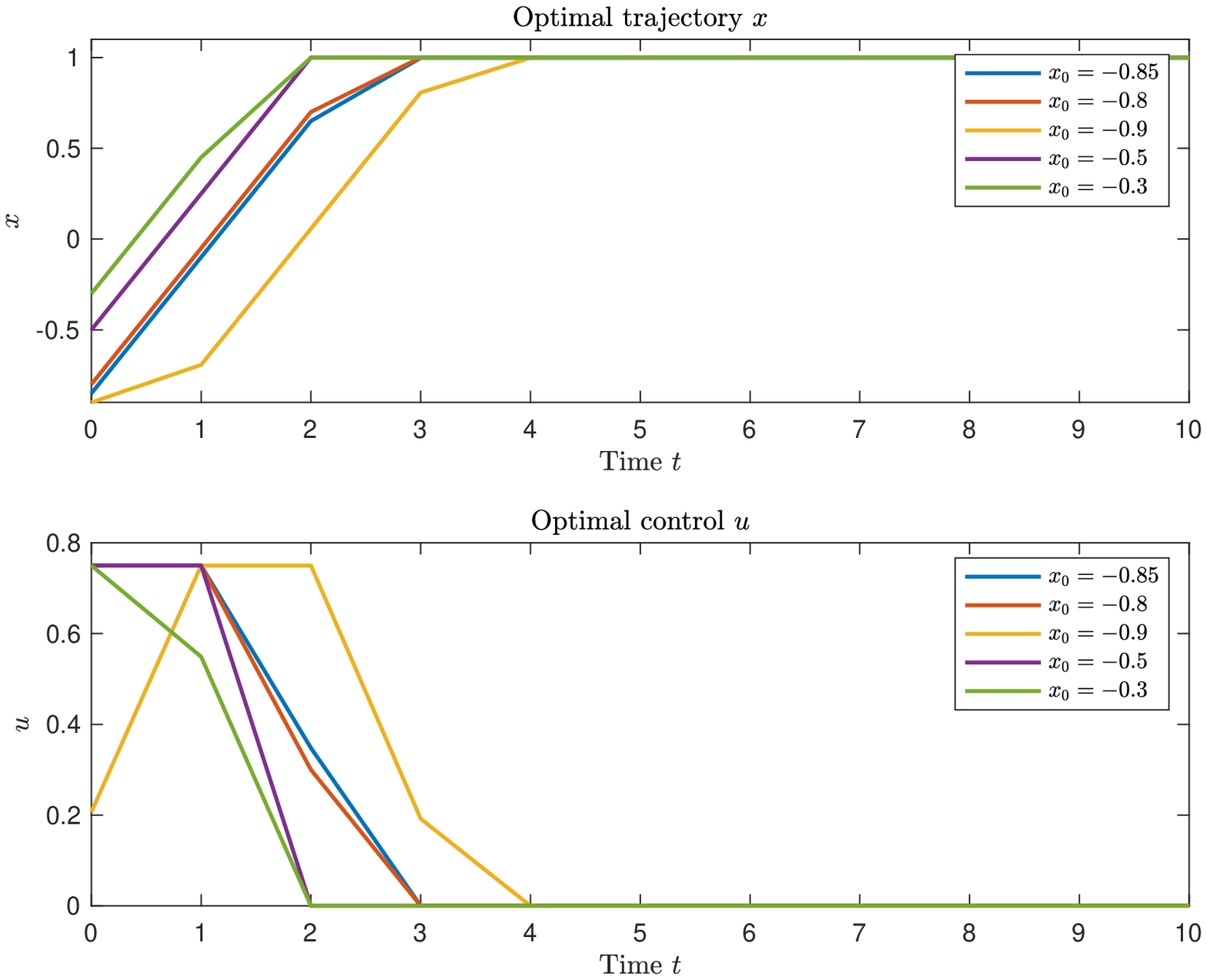}
	\end{minipage}
	\begin{minipage}[t]{0.47\textwidth}
	\includegraphics[width=\textwidth]{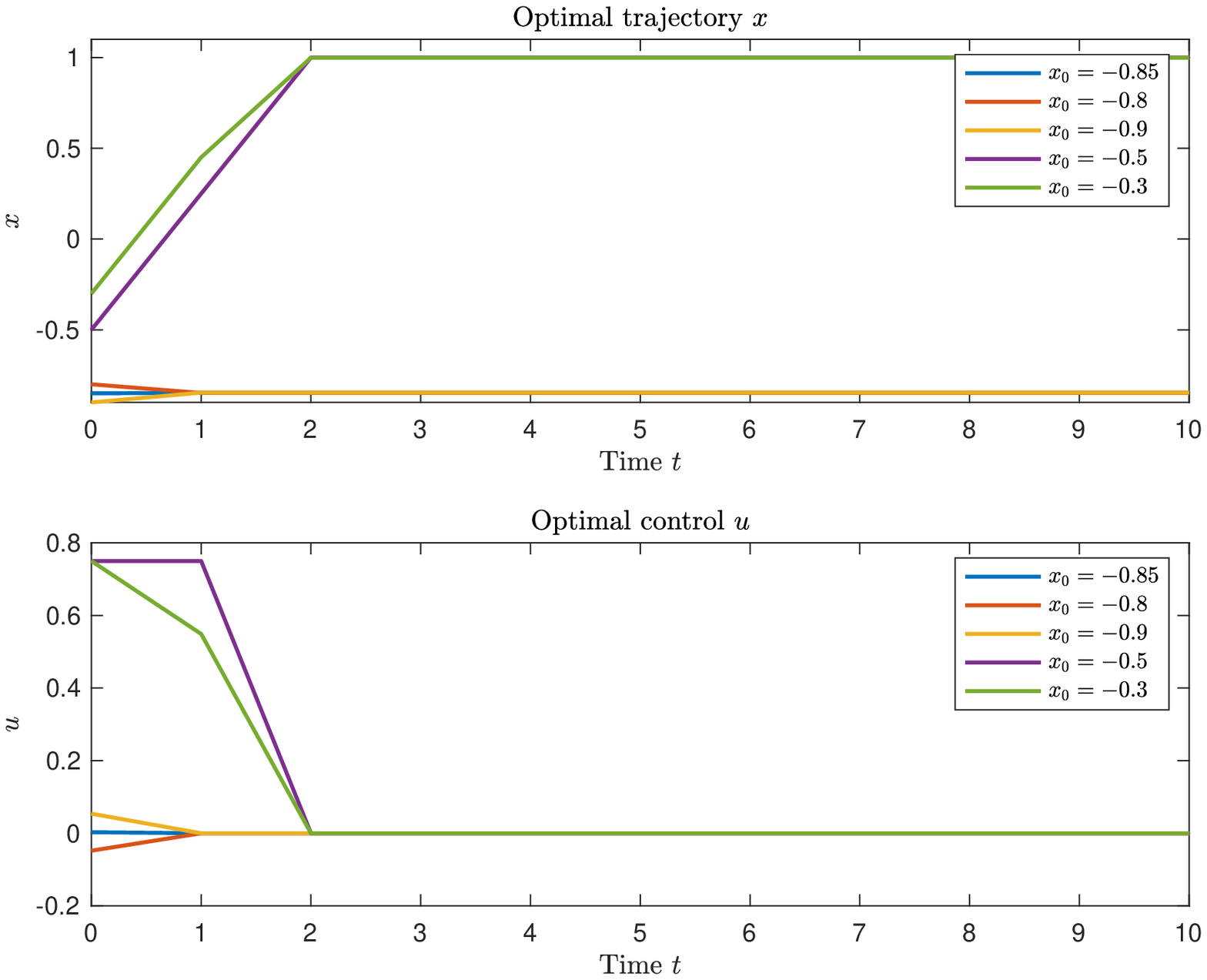}
	\end{minipage}
	\caption{Example \ref{ex:1} with $\beta = 0.7$ (left) and $\beta = 0.6$ (right) for different start values $x_0$}\label{fig: ex1 x0}
	\end{center}
\end{figure}

Figure \ref{fig: ex1 x0}, for fixed $\beta = 0.7$ we consider different initial values $x_0$. As we can see, the initial value determines to which equilibrium the trajectory converges. This underpins the theoretical results of Theorem \ref{th: locturnpike} and especially of Lemma \ref{lem: optimal}. We note that for a completely controllable system such a behaviour cannot occur in undiscounted problems.
\end{bsp}

The following modified example illustrates the case that the interval $(\beta_1, \beta_2)$ is empty.

\begin{bsp}\label{ex:2}
	Consider again the system $f(x,u)=x+u$, now with stage cost $\ell(x,u)=x^4-\frac 1 4 x^3 - \frac 7 4 x^2+\gamma |u|$ with $\gamma\neq 0$. As the added term has no influence on the conditions of Theorem \ref{th: locturnpike} we can again estimate $\beta_2\approx 0.67$. Further, for $\gamma=0$ we get the same stage cost as in Example \ref{ex:1} above. In contrast to Example \ref{ex:1}, now for $\gamma$ large enough we can observe that $(\beta_1, \beta_2)$ is empty. This fact is illustrated in Figure \ref{fig: ex2 beta} for $\gamma = 10$. For the numerical results we use the same setting as in example \ref{ex:1}.
	
	\begin{figure}[htb]
		\begin{center}
			\includegraphics[width=0.7\textwidth]{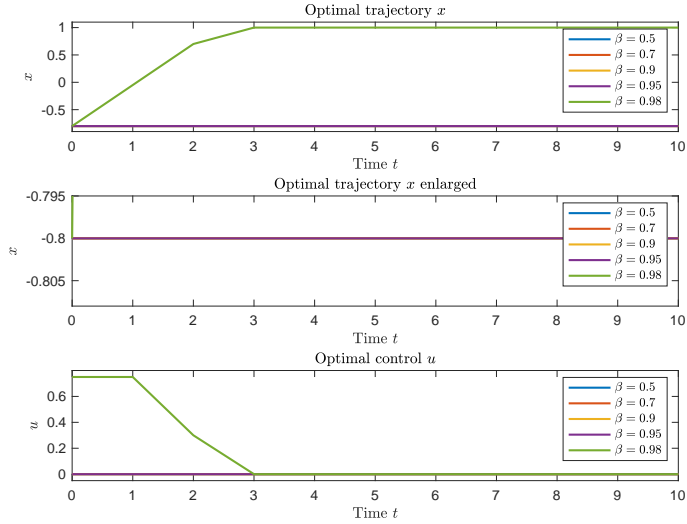}
			\caption{Example \ref{ex:2} with $\gamma = 10$ for different discount factor $\beta$}\label{fig: ex2 beta}
		\end{center}
	\end{figure}
	
In contrast, in the graph with $\gamma = 10$ we can clearly observe that independent of the discount factor $\beta$ we do not get convergence to the local equilibrium any more For $\beta$ large enough we even get convergence to the global equilibrium.
	\begin{figure}[htb]
		\begin{center}
		\begin{minipage}[t]{0.47\textwidth}
		\includegraphics[width=\textwidth]{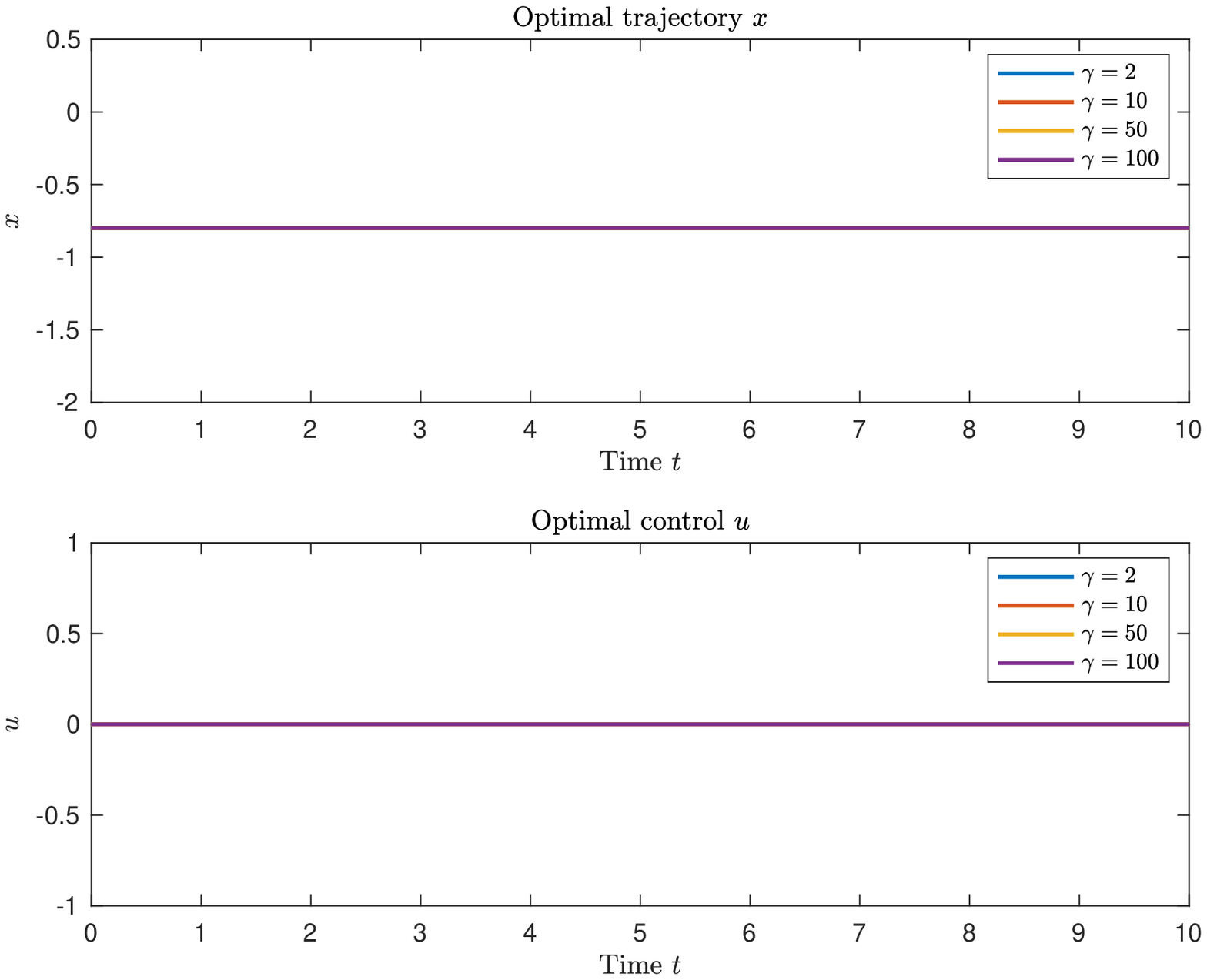}
	\end{minipage}
	\begin{minipage}[t]{0.47\textwidth}
		\includegraphics[width=\textwidth]{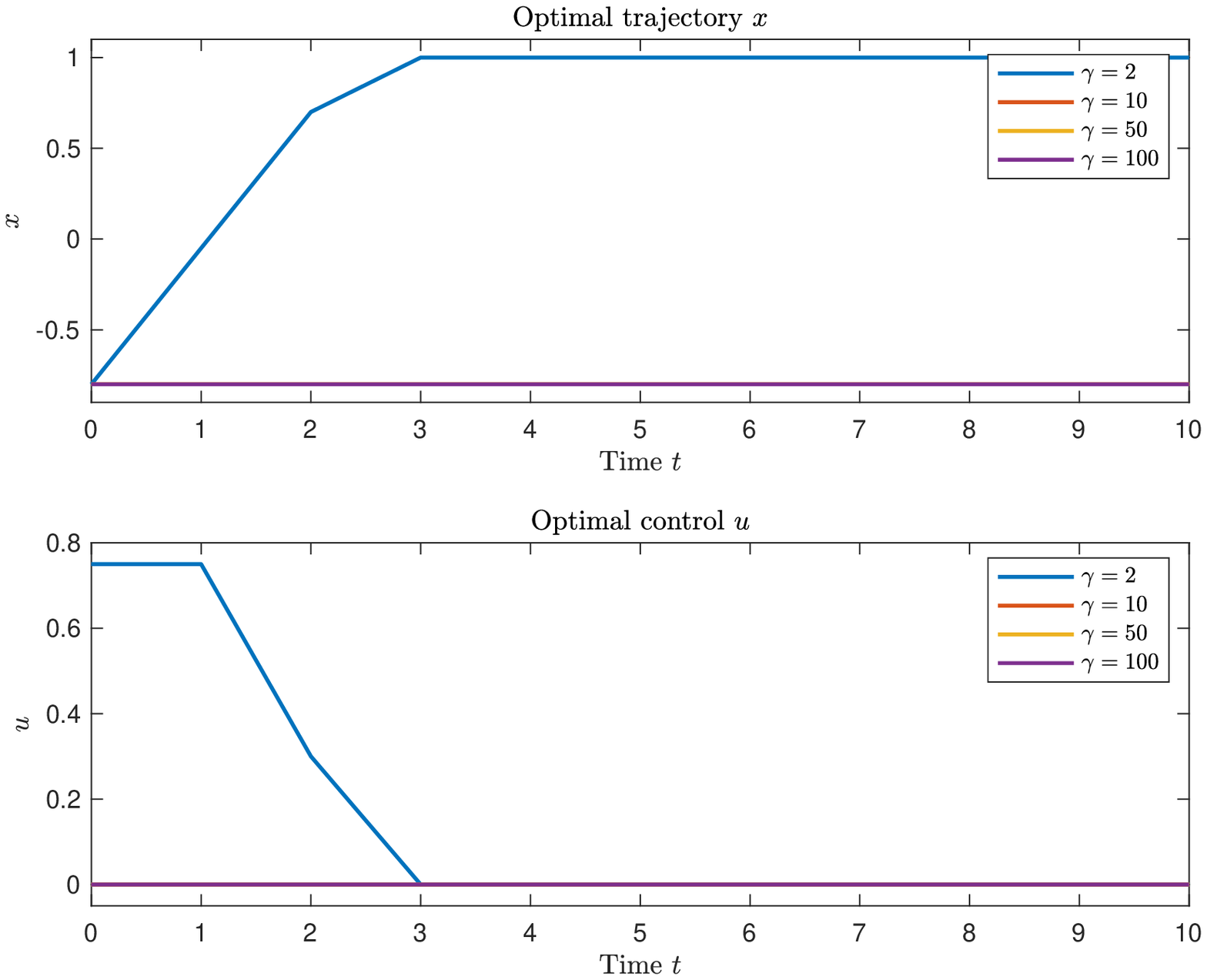}
		\end{minipage}
			\caption{Example \ref{ex:2} with $\beta = 0.7$ (left) and $\beta= 0.95$ (right) for different $\gamma$}\label{fig: ex2 gamma}
		\end{center} 
	\end{figure}

In order to examine this property in more detail we illustrate the behaviour of different values of $\gamma$ for fixed discount factors $\beta$ in Figure \ref{fig: ex2 gamma}. For $\gamma >1$ and $\beta \lesssim 0.95$ we can observe that the trajectories stay near by the start value and do not move away. In contrast, for $\beta \approx 1$ the trajectories converge to the global equilibrium. Thus, we do not get convergence to the local equilibrium any more
\end{bsp}

The two examples, above, have the particular feature that the dynamics is affine and the stage cost $\ell$ is strictly convex in a neighbourhood of the optimal equilibria. In this case, similar arguments as used in the proof of Theorem 4.2 in \cite{GMKW20} show that local strict dissipativity always holds. More precisely, we can restrict the proof of Theorem 4.2 in \cite{GMKW20} to a bounded neighbourhood $\X_{\NN}\subset \X$ of the local equilibrium $\xl$, e.g., $\BB_\eps(\xl)$, $\eps>0$, instead of $\X$, and a local strict convex stage cost function $\ell$. Following the proof, $\mathrm{D}\tell\equbl=0$ holds in the neighbourhood $\X_{\NN}$, which by the local strict convexity of $\tell$ implies that $\equbl$ is a strict local minimum. Together with the boundedness of $\X_{\NN}$, this implies the existence of $\alpha_\beta\in\KK_\infty$ and thus local discounted strict dissipativity. We remark that the calculation of $\lambda$ is the same as in the global case and yields a linear storage function. In the special case of Example \ref{ex:1}, above, it yields the storage function $\lambda\equiv 0$. In conclusion, local strict dissipativity always holds if the dynamics is affine and the stage cost $\ell$ is strictly convex near the locally optimal equilibrium.

With this observation, our dissipativity based analysis provides a complementary approach to the stable manifold based analysis carried out, e.g., in \cite{HKHF03}. Particularly, we can conclude that the model from this reference exhibits two equilibria at which the local turnpike property holds, which explains why the optimal trajectories are correctly reproduced by nonlinear model predictive control as shown in \cite[Section 5.1]{Gruene2015}.

Our final example demonstrates that strict convexity of $\ell$ is not needed for obtaining strict dissipativity, thus showing that a dissipativity based analysis allows for strictly weaker assumptions than strict convexity of $\ell$.

\begin{bsp}\label{ex:nonconvex}
	Consider the 1d control system
\[ x^+ = f(x,u) = 2x+u \]
with state constraints $\X=[-1,1]$, control constraints $\U=[-3,3]$, and stage cost
\[ \ell(x,u) = -x^2/2 + u^2.\]
Obviously, the stage cost is strictly concave in $x$ and strictly convex in $u$. Nevertheless, we can establish discounted strict $(x,u)$-dissipativity in $(x^*,u^*)=(0,0)$ (in this example even global) for $\beta\ge3/5$ with $\lambda(x) = -x^2$. This follows from the fact that with $a=2\beta/\sqrt{1+\beta}$ and $b=\sqrt{1+\beta}$ we have 
\begin{eqnarray*}
\ell(x,u) + \lambda(x) - \beta\lambda(f(x,u)) & = & 
-x^2/2 + u^2 - x^2 + \beta (2x+u)^2\\
& = & (4\beta-3/2) x^2 + 4\beta xu + (1+\beta)u^2\\
& = & (a x + b u)^2 + \left(4\beta-\frac{3}{2} - \frac{4\beta^2}{1+\beta}\right)x^2\\
& \ge & (a x + b u)^2,
\end{eqnarray*}
where the last inequality holds since the term in the large brackets is $\ge 0$ for $\beta \ge 3/5$.

Since the system is completely controllable in finite time, hence exponentially stabilizable, Theorem \ref{th: disturninf} in conjunction with Remark \ref{rem:suffcond}(ii) implies that for sufficiently large $\beta$ turnpike behaviour occurs at $x^*=0$. This is confirmed for $\beta=0.7$ in the left graph in Figure \ref{fig: ex nonconvex}. In contrast to this, the right graph in Figure~\ref{fig: ex nonconvex} shows that for $\beta=0.6$ the turnpike behaviour for $x^*=0$ does not occur. Rather, the optimal solution converges to the upper bound $x=1$ of the state constraint set. In this example, the numerical computations indicate that $\beta = 3/5=0.6$ is a relatively precise estimate of the threshold for the occurrence of the turnpike property at $x^*=0$, although for $\beta$ decreasing from $0.7$ to $0.6$ the set of initial values around $x^*=0$ for which the turnpike behaviour can be seen shrinks down rapidly.

	\begin{figure}[htb]
	\begin{center}
		\begin{minipage}[t]{0.47\textwidth}
			\includegraphics[width=\textwidth]{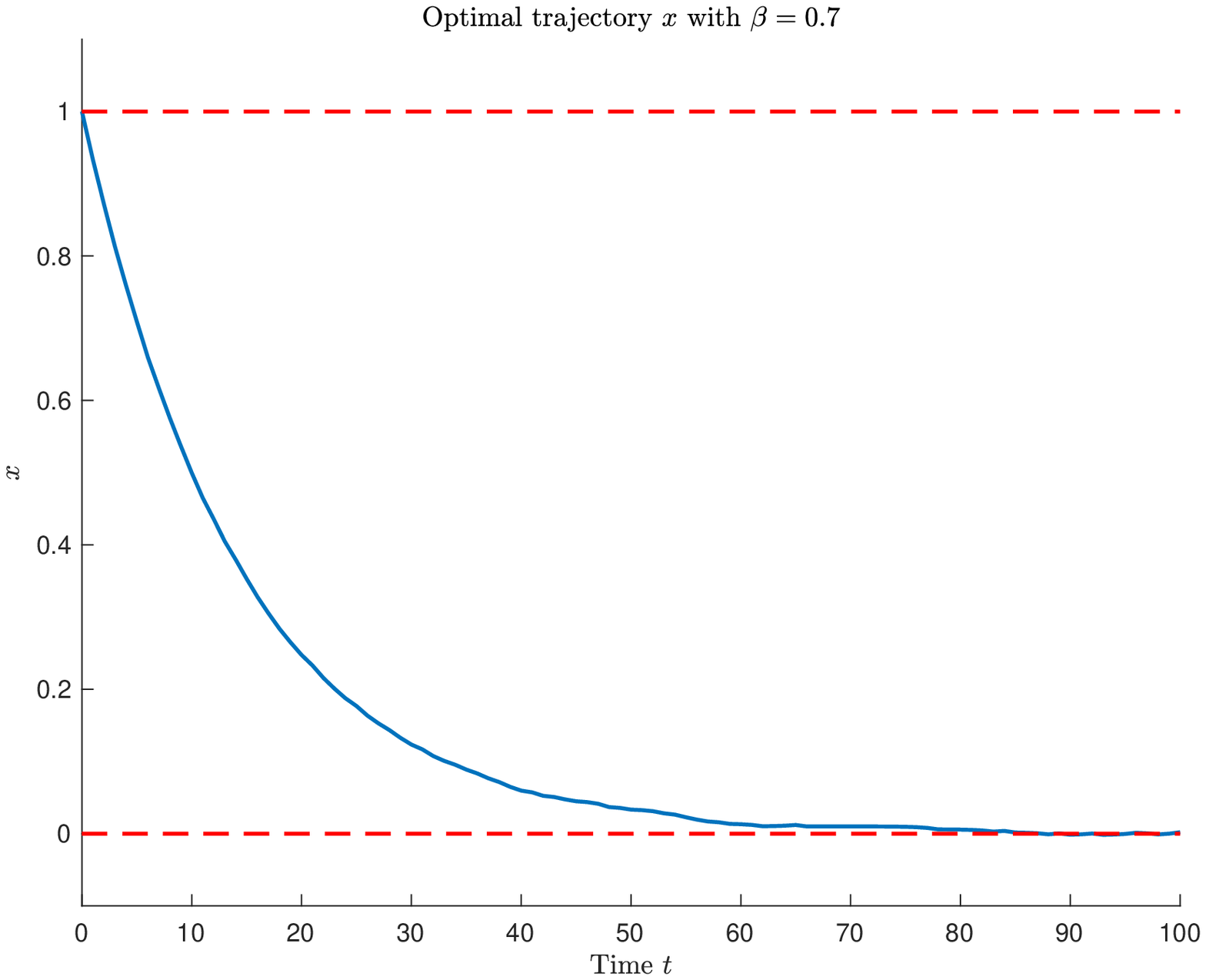}
		\end{minipage}
		\begin{minipage}[t]{0.47\textwidth}
			\includegraphics[width=\textwidth]{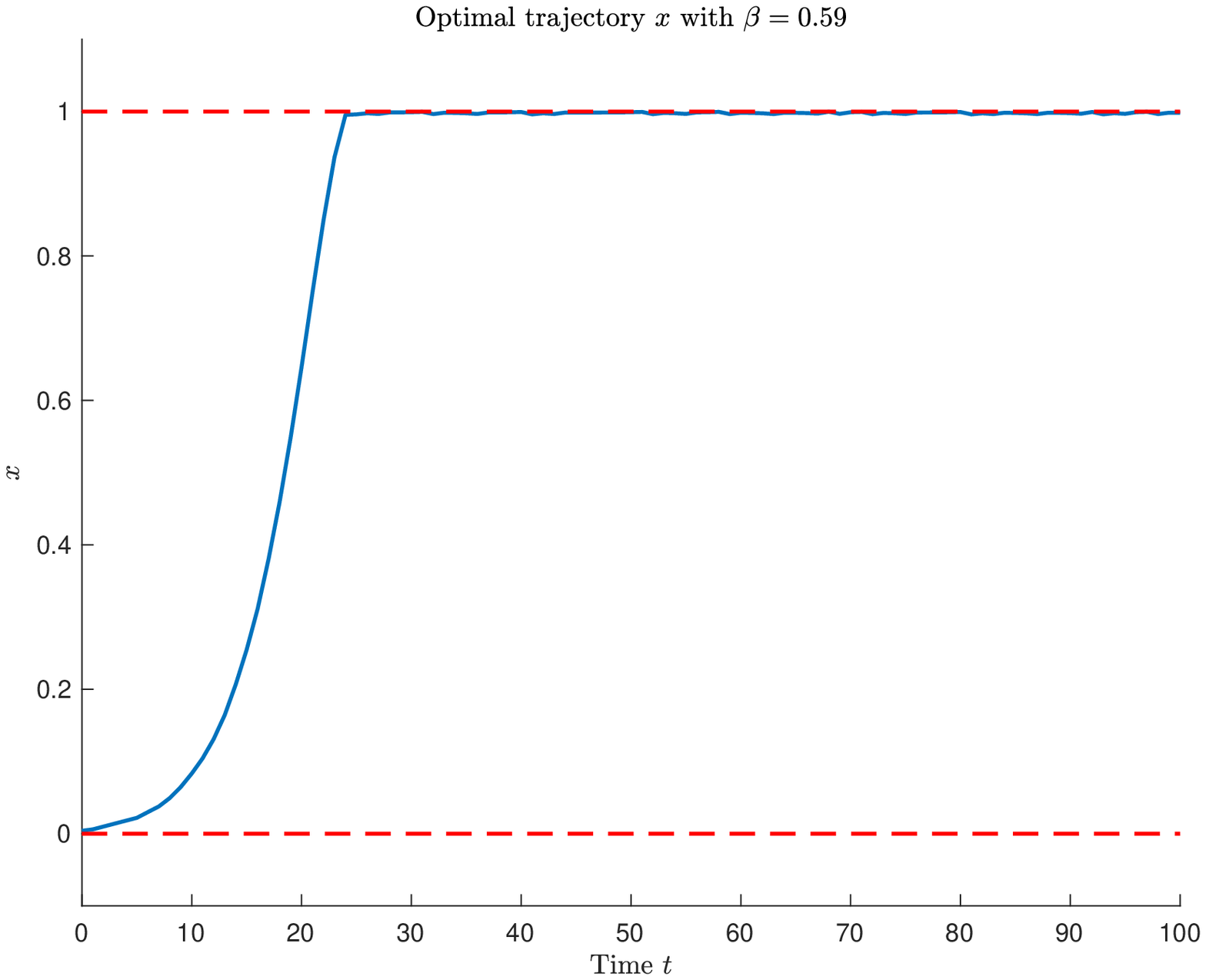}
		\end{minipage}
		\caption{Optimal trajectories for Example \ref{ex:nonconvex} with $\beta = 0.7$ and $x_0=1$ (left) and with $\beta= 0.59$ and $x_0=0.004$ (right)}\label{fig: ex nonconvex}.
	\end{center} 
\end{figure}
\end{bsp}

\section{Conclusion}\label{sec:conclusion}

In this paper we have shown that a local strict dissipativity assumption in conjunction with an appropriate growth condition on the optimal value function can be used in order to conclude a local turnpike property at an optimal equilibrium.  The turnpike property holds for discount factors from an interval $[\beta_1,\beta_2]$, where $\beta_1$ is determined by local quantities while $\beta_2$ is also determined by properties of the optimal control problem away from the local equilibrium. Hence, local and global properties together determine whether the interval is not empty. This is in accordance with other approaches for analysing local stability of equilibria in discounted optimal control such as those based on stable and unstable manifolds \cite{HKHF03}. In contrast to other approaches, however, the dissipativity based approach is not limited to (locally) strictly convex problems, as our last example showed.

{\bibliographystyle{plain}
\bibliography{references}}

\end{document}